\documentclass[a4paper,12pt]{article}
\usepackage[T1]{fontenc}
\usepackage[english]{babel}
\usepackage[dvips]{graphicx}
\usepackage{amssymb,amsthm,amsmath}
\usepackage{pstricks}

\newcommand{\re}{\textnormal{Re}\,}

\renewcommand{\span}{\textnormal{span}\,}
\newcommand{\e}{\varepsilon}
\newcommand{\p}{\varphi}

\newcommand{\PP}{ { \R^2 \setminus\{ 0 \} } }
\newcommand{\PS}{ { \Rn \setminus\{ 0 \} } }
\newcommand{\R}{\mathbb{R}}
\newcommand{\C}{\mathbb{C}}
\newcommand{\Rn}{ {\mathbb{R}^n} }

\newcommand{\comment}[1]{}

\newtheorem{theorem}[equation]{Theorem}
\newtheorem{lemma}[equation]{Lemma}

\newtheorem{corollary}[equation]{Corollary}

\theoremstyle{definition}
\newtheorem{remark}[equation]{Remark}
\newtheorem{definition}[equation]{Definition}

\numberwithin{equation}{section}

\newcounter{minutes}
\divide\time by 60
\newcounter{hours}
\multiply\time by 60 \addtocounter{minutes}{-\time}
\begin{document}
\psset{linewidth=1pt}

\begin{center}
{\Large \bf Local Convexity Properties of Quasihyperbolic Balls in Punctured Space}
\end{center}
\medskip

\begin{center}
{\large Riku Klén} \footnote{\noindent Department of Mathematics, University of Turku, FIN-20014, FINLAND\\
e-mail: riku.klen@utu.fi, phone: +358 2 333 6013, fax: +358 2 333 6595}
\end{center}
\bigskip

\begin{abstract}
  This paper deals with local convexity properties of the quasihyperbolic metric in the punctured space. We consider convexity and starlikeness of quasihyperbolic balls.

  \hspace{5mm}


\end{abstract}

\comment{
\bigskip
\begin{center}
\texttt{File:~\jobname .tex, 
         printed: \number\year-\number\month-\number\day,
         \thehours.\ifnum\theminutes<10{0}\fi\theminutes}
\end{center}
}

%%%%%%%%%%%%%%%%%%%%%%%%%%%%%%%%%%%%%%%%%%%%%%%%%%%%%%%%%%%%%
\section{Introduction}

The \emph{quasihyperbolic distance} between two points $x$ and $y$ in a proper subdomain $G$ of the Euclidean space $\Rn$, $n \ge 2$, is defined by
$$
  k_G(x,y) = \inf_{\alpha \in \Gamma_{xy}} \int_{\alpha}\frac{|dz|}{d(z,\partial G)},
$$
where $d(z,\partial G)$ is the (Euclidean) distance between the point $z \in G$ and the boundary of $G$ and $\Gamma_{xy}$ is the collection of all rectifiable curves in $G$ joining $x$ and $y$.

Since its introduction by F.W. Gehring and B.P. Palka \cite{gp} in 1976, the quasihyperbolic metric has been widely applied in geometric function theory and mathematical analysis in general, see e.g. \cite{vu3,v1}. Quasihyperbolic geometry has recently been studied by P. Hästö \cite{h} and H. Lindén \cite{l}.

The purpose of this paper is to study the metric space $(G,k_G)$ and especially local convexity properties of \emph{quasihyperbolic balls} $D_G(x,M)$ defined by
$$
  D_G(x,M) = \{ z\in G \colon k_G(x,z)<M \}.
$$
In the dimension $n=2$ we call these balls \emph{disks} and we often identify $\R^2$ with the complex plane $\C$.

M. Vuorinen suggested in \cite{vu4} a general question about the convexity of balls of small radii in metric spaces. Our work is motivated by this question and our main result Theorem \ref{mainthm} provides an answer in a particular case. For the definition of starlike domains see Definition \ref{defstarlikeness}.

\begin{theorem}\label{mainthm}
  1) For $x \in \PS$ the quasihyperbolic ball $D_\PS(x,M)$ is strictly convex for $M \in (0,1]$ and it is not convex for $M > 1$.

  \noindent 2) For $x \in \PS$ the quasihyperbolic ball $D_\PS(x,M)$ is strictly starlike with respect to $x$ for $M \in (0,\kappa]$ and it is not starlike with respect to $x$ for $M > \kappa$, where $\kappa$ is defined by (\ref{kappa}) and has a numerical approximation $\kappa \approx 2.83297$.
\end{theorem}

Theorem \ref{mainthm} in the case $n = 2$ is illustrated in Figure \ref{fig1}. O. Martio and J. Väisälä \cite{mv} have recently proved that if $G$ is convex then $D_G(x,M)$ is also convex for all $x \in G$ and $M > 0$.

\begin{figure}[ht!]
  \begin{center}
    \includegraphics[width=5cm]{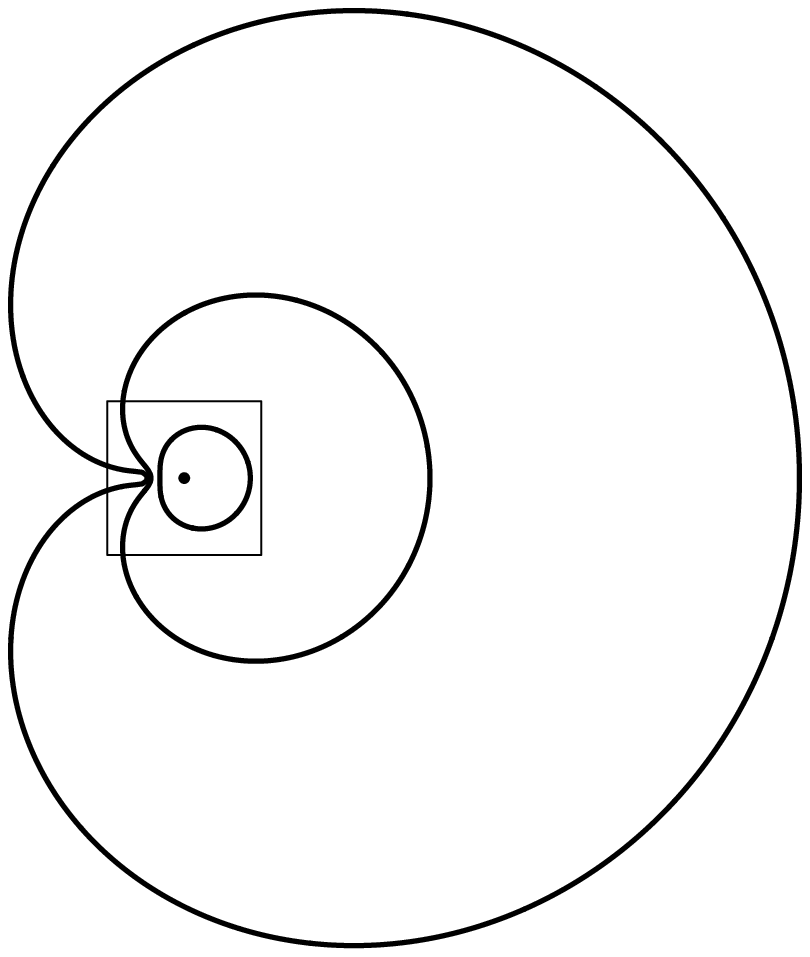}\hspace{1cm}
    \includegraphics[width=5cm]{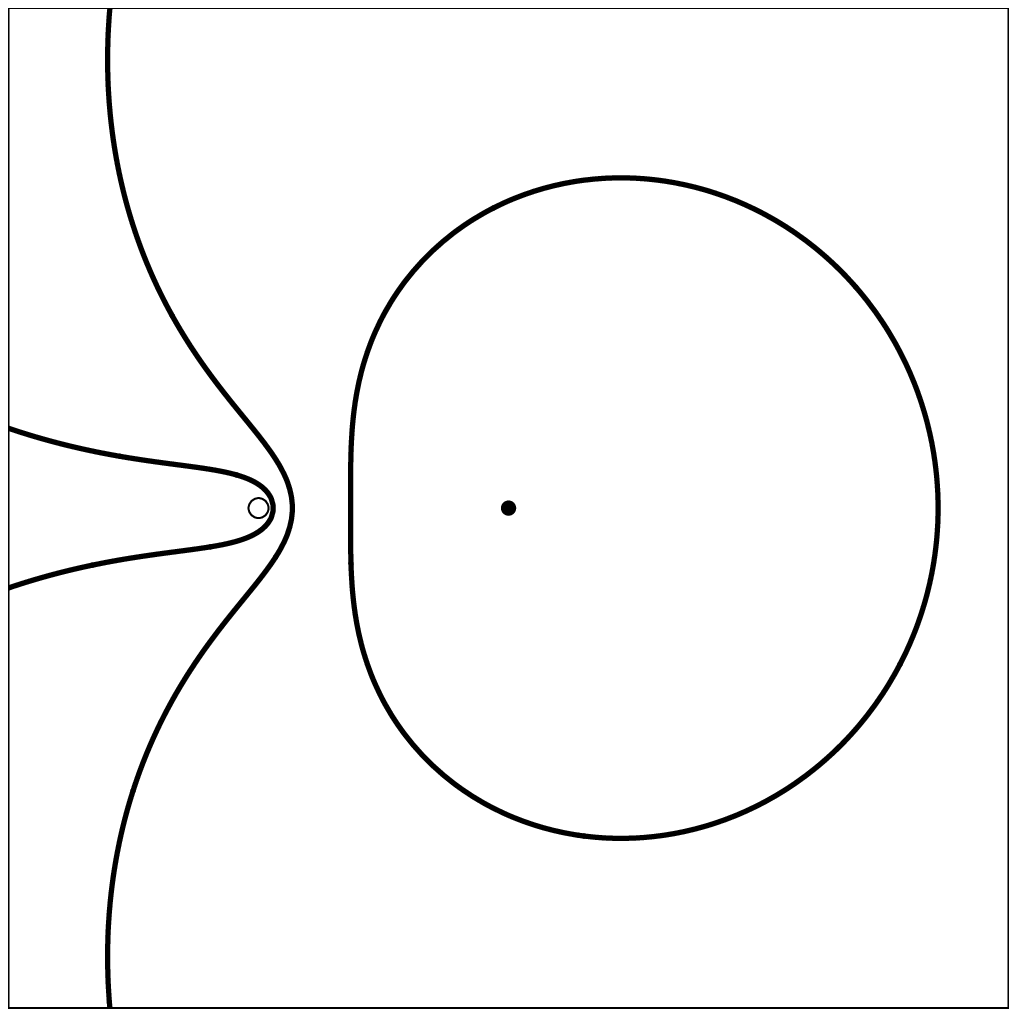}
    \caption{Boundaries of quasihyperbolic disks $D_\PP(x,M)$ with radii $M=1$, $M=2$ and $M=\kappa$.\label{fig1}}
  \end{center}
\end{figure}

%%%%%%%%%%%%%%%%%%%%%%%%%%%%%%%%%%%%%%%%%%%%%%%%%%%%%%%%%%%%%
\section{Quasihyperbolic balls with large and small radii}

In this section we consider the behavior of quasihyperbolic balls with large and small radii.

Let us define $\phi$-uniform domains, which were introduced by M. Vuorinen \cite[2.49]{vu2}, and consider quasihyperbolic balls with large radii in $\phi$-uniform domains. We use notation $m(a,b) = \min \{ d(a), d(b) \}$, where $d(x) = d(x,\partial G)$.

\begin{definition}
  Let $\phi \colon [0,\infty) \to [0,\infty)$ be a continuous and strictly increasing homeomorphism. Then a domain $G \subsetneq \Rn$ is \textit{$\phi$-uniform} if
  $$
    k_G(x,y) \le \phi \left( \frac{|x-y|}{m(x,y)} \right)
  $$
  for all $x,y \in G$.
\end{definition}

\begin{lemma}\label{phiproperty}
Fix $\phi$, let $G$ be $\phi$-uniform, $x_0 \in G$ and $M > 0$. If $x \in G$ with $m(x,x_0) > |x-x_0|/\phi^{-1}(M)$ then $x \in D_G(x_0,M)$.
\end{lemma}
\begin{proof}
  Since $\phi$ is a homeomorphism $m(x,x_0) > |x-x_0|/\phi^{-1}(M)$ implies
  $$
    \phi \left( \frac{|x-x_0|}{m(x,x_0)} \right) < M
  $$
  and since $G$ is $\phi$-uniform
  $$
    k_G(x,x_0) \le \phi \left( \frac{|x-x_0|}{m(x,x_0)} \right) < M.
  $$
  Therefore $x \in D_G(x_0,M)$.
\end{proof}

\begin{definition}
  Let $\delta \in (0,1)$ and $r_0 > 0$ be fixed and $G \subset \Rn$ be a bounded domain. We say that $G$ satisfies the \textit{$(\delta,r_0)$-condition} if for all $z \in \partial G$ and $r \in (0,r_0]$ there exists $x \in B^n(z,r) \cap G$ such that $d(x) > \delta r$.
\end{definition}

\begin{theorem}\label{phideltar0}
  Assume $G$ is a bounded $\phi$-uniform domain and satisfies the $(\delta,r_0)$-condition for a fixed $\delta \in (0,1)$ and $r_0 > 0$. Let us assume $r_1 \in(0,r_0)$ and fix $x_0 \in G$ and $z \in \partial G$. Then $d \big( D_G(x_0,M),z \big) < r_1$ for
  \begin{equation}
    M > \phi \left( \frac{|x_0-z|+r_2}{\delta r_2} \right),\label{phidelta}
  \end{equation}
  where $r_2 = \min \{ r_1,d(x_0)/2 \}$.
\end{theorem}
\begin{proof}
  Since $G$ satisfies the $(\delta,r_0)$-condition and $r_2 < r_0$ we can choose $x \in B^n(z,r_2) \cap G$ with $d(x) > \delta r_2$. Now
  $$
    m(x_0,x) = \min \{ d(x_0),d(x) \} = d(x) > \delta r_2
  $$
  and $|z-x| < r_2$. The inequality (\ref{phidelta}) is equivalent to
  $$
    \delta r_2 > \frac{|x_0-z|+r_2}{\phi^{-1}(M)}.
  $$
  Since $|z-x| < r_2$ and by the triangle inequality
  $$
    \frac{|x_0-z|+r_2}{\phi^{-1}(M)} > \frac{|x_0-z|+|z-x|}{\phi^{-1}(M)} \ge \frac{|x_0-x|}{\phi^{-1}(M)}.
  $$
  Now we have
  $$
    m(x_0,x) > \delta r_2 > \frac{|x_0-z|+r_2}{\phi^{-1}(M)} > \frac{|x_0-x|}{\phi^{-1}(M)}
  $$
  and by Lemma \ref{phiproperty} we have $x \in G \cap D_G(x_0,M)$. Therefore
  $$
    d \big( D_G(x_0,M),z \big) \le |z-x| < r_2 \le r_1
  $$
  and the claim is clear.
\end{proof}

\begin{corollary}
  Let $G \subset \Rn$ be a bounded $\phi$-uniform domain and let $G$ satisfy the $(\delta,r_0)$-condition. For a fixed $s \in (0,r_0)$ and $x \in G$ there exists a number $M(s)$ such that
  $$
    \overline{G} \subset D_G \big( x,M(s) \big) + B^n(s) = \left\{ y+z \colon y \in D_G \big( x,M(s) \big) , \, |z| < s \right\}.
  $$
\end{corollary}
\begin{proof}
  We choose
  $$
    M(s) > \max_{z \in \partial G} \phi \left( \frac{|x-z|+r}{\delta r} \right),
  $$
  where $r = \min \{ s,d(x)/2 \}$. By Theorem \ref{phideltar0} the assertion follows.
\end{proof}

Let us then point out that quasihyperbolic balls of small radii become more and more like Euclidean balls when the radii tend to zero. We shall study the local structure of the boundary of a quasihyperbolic ball and show that the boundary is round from the inside and cannot have e.g. outwards directed conical parts.

\begin{definition}
  Let $\gamma$ be a curve in domain $G \subsetneq \Rn$. If
  $$
    k_G(x,y) + k_G(y,z) = k_G(x,z)
  $$
  for all $x,z \in \gamma$ and $y \in \gamma'$, where $\gamma'$ is the subcurve of $\gamma$ joining $x$ and $z$, then $\gamma$ is a \textit{geodesic segment} or briefly a \textit{geodesic}. We denote a geodesic between $x$ and $y$ by $J_k[x,y]$.
\end{definition}

\begin{theorem}
For a proper subdomain $G$ of $\Rn$, $M>0$ and $y\in
\partial D_G(x,M)$, let $J_k[x,y]$ be a geodesic segment of the
quasihyperbolic metric joining $x$ and $y$. For $z\in J_k[x,y]$ we
have
$$
  B^n\left(z,\frac{|z-y|}{1+u}\right) \subset D_G(x,M),
$$
where $u=|z-y|/d(z)$.
\end{theorem}

\begin{proof}
  By \cite[Lemma 1]{go} there exists $J_k[x,y]$. By the choice of $z$ we have
  $$
    M = k_G(x,y) = k_G(x,z)+k_G(z,y)
  $$
  and by the triangle inequality for $w\in D_G\big(z,k_G(z,y)\big)$ we have
  $$
    k_G(x,w)\le k_G(x,z)+k_G(z,w) < M.
  $$
  Now
  $$
    D_G\big(z,k_G(z,y)\big) \subset D_G(x,M).
  $$
  By \cite[page 347]{vu1}
  $$
    B^n \left( z, \left( 1-e^{-k_G(z,y)} \right) d(z) \right) \subset D_G\big(z,k_G(z,y)\big)
  $$
  and therefore
  $$
    B^n \left( z, \left( 1-e^{-k_G(z,y)} \right) d(z) \right) \subset D_G(x,M).
  $$
  By \cite[Lemma 2.1]{gp} $k_G(z,y)\ge \log\left( 1+\frac{|z-y|}{d(z)} \right)$ and therefore
  \begin{eqnarray*}
    \left( 1-e^{-k_G(z,y)} \right) d(z) & \ge & \left(1-\frac{d(z)}{d(z)+|z-y|}\right)d(z)\\
    & = & \frac{|z-y|}{1+u}
  \end{eqnarray*}
  for $u=\frac{|z-y|}{d(z)}$. Now
  $$
    B^n\left(z,\frac{|z-y|}{1+u}\right) \subset B^n \left( z, \left( 1-e^{-k_G(z,y)} \right) d(z) \right)
  $$
  and the claim is clear.
\end{proof}

Now we have found a Euclidean ball $B^n(z,r)$ inside the quasihyperbolic ball $D_G(x,M)$ with the following property:
$$
  \frac{r}{d\big(z,\partial D_G(x,M)\big)} \rightarrow 1, \textrm{ when } z
  \rightarrow \partial D_G(x,M).
$$
Geometrically this convergence means that the boundary of the quasihyperbolic ball must be round from the interior. The boundary cannot have any cone shaped corners pointing outwards from the ball. However, there can be corners in the boundary pointing inwards to the ball. An example in $\PP$ is the quasihyperbolic disk with $M > \pi$. This example is considered in more detail in Remark \ref{notround}.

\begin{definition}\label{defstarlikeness}
  Let $G \subset \Rn$ be a domain and $x \in G$ . We say that $G$ is \textit{starlike with respect to $x$} if each line segment from $x$ to $y \in G$ is contained in $G$. The domain $G$ is \textit{strictly starlike with respect to $x$} for $x \in G$ if $G$ is bounded and each ray from $x$ meets $\partial G$ at exactly one point.
\end{definition}

The following result considers starlikeness of quasihyperbolic balls in starlike domains. The same result was independently obtained by J. Väisälä \cite{v2}.

\begin{theorem}\label{starshapedkinssd}
  If $G \subsetneq \Rn$ is a starlike domain with respect to $x$, then the quasihyperbolic ball $D_G(x,M)$ is starlike with respect to $x$.
\end{theorem}
\begin{proof}
  We need to show that the function $f(y) = k_G(x,y)$ is increasing along each ray from $x$ to $\partial G$. To simplify notation we may assume $x=0$.

  Let $y \in G \setminus \{ x \}$ be arbitrary and denote a geodesic segment from $x$ to $y$ by $\gamma$. Let us choose any $y' \in (x,y)$ and denote
  $$
    \gamma' = \frac{|y'|}{|y|}\gamma = c \gamma.
  $$
  Since $G$ is starlike with respect to $x$ the path $\gamma'$ from $x$ to $y'$ is in $G$. Therefore
  $$
    k_G(x,y') \le \int_{\gamma'}\frac{|dz|}{d(z)} = \int_{\gamma} \frac{c |dz|}{d(c z)}.
  $$
    Since $G$ is starlike with respect to $x$ we have $d(c z) \ge c d(z)$ which is equivalent to
  $$
    \frac{c}{d(c z)} \le \frac{1}{d(z)}.
  $$
  Now
  $$
    k_G(x,y') \le \int_\gamma \frac{c |dz|}{d(c z)} \le \int_\gamma \frac{|dz|}{d(z)} = k_G(x,y)
  $$
  and $f$ is increasing along each ray from $x$ to $\partial G$.
\end{proof}

For a domain $G \subset \Rn$ and quasihyperbolic ball $D_G(x,M)$, $x \in G$ and $M > 0$, we define \emph{the points that can affect the shape of $D_G(x,M)$} to be the set
$$
  \{ z \in \partial G \colon |z-y| = d(y) \textnormal{ for some } y \in D_G(x,M) \}.
$$

Let $G$ be a domain and fix $x \in G$ and $M > 0$. Now by \cite[page 347]{vu1} we know that $D_G(x,M) \subset B^n(x,Rd(x))$, for $R = e^M-1$, and therefore for each $y \in D_G(x,M)$ we have $d(y) \le d(x)+2Rd(x) = d(x)(2e^M-1)$. This fact is generalized in the following lemma.

\begin{lemma}\label{sslemma}
 Let $G \subsetneq \Rn$ be a domain, $x \in G$ and $y \in \partial G$. Then the points that can affect the shape of the quasihyperbolic ball $D_G(x,M)$ for $M \in (0,1]$ are in the closure of the set
  $$
    U_y = B^n \big( x,|x-y|(2e^M-1) \big) \setminus \left\{ z \in \Rn \setminus \{ y \} \colon \measuredangle{x' y z} \le \pi/2-1, x' = 2y-x \right\},
  $$
  where $\measuredangle{x' y z}$ is the angle between line segments $[x',y]$ and $[z,y]$ at $y$.
\end{lemma}
\begin{proof}
  Let us consider $G' = \Rn \setminus \{ y \}$. Now $G \subset G'$ and therefore $D_G(x,M) \subset D_{G'}(x,M)$. Now the points that can affect the shape of $D_G(x,M)$ need to be inside $\overline{B^n \big( x,|x-y|(2e^M-1) \big) }$.

  Let $z \in \partial D_{G'}(x,M)$. Because $M \le 1$ we have by (\ref{moeq}) $\measuredangle{xyz} \le 1$. Therefore the points in
  $$
    \left\{ z \in \Rn \setminus \{ y \} \colon \measuredangle{x' y z} \le \pi/2-1, x' = 2y-x \right\}
  $$
  do not affect the shape of $D_{G'}(x,M)$. Since $D_G(x,M) \subset D_{G'}(x,M)$, the claim is clear.
\end{proof}

\begin{theorem}\label{sstheorem}
  For a domain $G \subsetneq \Rn$, $M \in (0,1]$ and $x \in G$ the quasihyperbolic ball $D_G(x,M)$ is starlike with respect to $x$.
\end{theorem}
\begin{proof}
  We denote
  $$
    V_x = G \cap \left( \bigcap_{y \in \partial G} U_y \right).
  $$
  The set $\overline{V_x}$ contains all of the boundary points of $G$ that affect the shape of $D_G(x,M)$. Therefore for fixed $x \in G$ we have $D_G(x,M) = D_{V_x}(x,M)$ and $D_G(x,M)$ is starlike with respect to $x$ by Theorem \ref{starshapedkinssd}, because $V_x$ is starlike with respect to $x$.
\end{proof}

\begin{remark}
  In Lemma \ref{sslemma} and Theorem \ref{sstheorem} we could replace $M \in (0,1]$ by $M \in (0,\alpha]$ and $\measuredangle{x' y z} \le \pi/2-1$ by $\measuredangle{x' y z} \le \pi/2-\alpha$ for any $\alpha \in [1,\pi/2)$. This modified version of Theorem \ref{sstheorem} was also proved by J. Väisälä \cite[Theorem 3.11]{v1}.
\end{remark}

%%%%%%%%%%%%%%%%%%%%%%%%%%%%%%%%%%%%%%%%%%%%%%%%%%%%%%%%%%%%%
\section{Convexity of quasihyperbolic balls in punctured space}

The set $\Rn \setminus \{ z \}$, $z \in \Rn$, is called a punctured space. To simplify notation we may assume $z = 0$. In this section we will find values $M$ such that the quasihyperbolic ball $D_\PS(x,M)$ is convex for all $x \in \PS$.

Let us assume that $x,y \in \PS$ and that the angle $\p$ between segments $[0,x]$ and $[0,y]$ satisfies $0 < \p \le \pi$. It can be shown
\cite[page 38]{mo} that
\begin{equation}\label{moeq}
  k_\PS(x,y) = \sqrt{\p^2+\log^2\frac{|x|}{|y|}}.
\end{equation}

In particular, we see that $k_\PS(x,y) = k_\PS(x,y_1)$, where $y_1$ is obtained from $y$ by the inversion with respect to $S^{n-1}(|x|)$, i.e. $y_1 = y|x|^2/|y|^2$. Hence this inversion maps the quasihyperbolic sphere $\{ z\in \PS \colon k_\PS(x,z)=M \}$ onto itself.

Quasihyperbolic balls are similar in $\PS$ for fixed $M$. In other words any quasihyperbolic ball of radius $M$ can be mapped onto any other quasihyperbolic ball of radius $M$ by rotation and stretching.

We will first consider convexity of the quasihyperbolic disks in the punctured plane $\PP$ and then extend the results to the punctured space $\PS$.

By (\ref{moeq}) we have a coordinate representation in the case $n = 2$
\begin{equation}\label{equation}
  x = (|x| \cos{\p},|x| \sin{\p}) = \left(e^{\pm\sqrt{M^2-\p^2}} \cos{\p},e^{\pm\sqrt{M^2-\p^2}} \sin{\p}\right),
\end{equation}
for $x\in \partial D_\PP(1,M)$ and $-M \leq \p \leq M$. By using this presentation we will prove the following result.

\begin{theorem}\label{notconvexPP}
  For $M>1$ and $z \in \PP$ the quasihyperbolic disk $D_\PP(z,M)$ is not convex.
\end{theorem}
\begin{proof}
  We may assume $z=1$ and let $x \in \partial D_\PP(z,M)$ be arbitrary. Assume $M>1$. By (\ref{equation}) we have
  $$
    x = \left(e^{\pm\sqrt{M^2-\p^2}} \cos{\p},e^{\pm\sqrt{M^2-\p^2}} \sin{\p}\right),
  $$
  where $-M \leq \p \leq M$.

  If $M > \pi/2$, then the claim is clear by symmetry because $\re x = e^{-M} > 0$ for $\p = 0$ and $\re x < 0$ for $\p = \pm M$.

  We will show that the function
  $$
    f(\p) = e^{-\sqrt{M^2-\p^2}}\cos{\p}
  $$
  is concave in the neighborhood of $\p = 0$ and the function
  $$
    g(\p) = e^{-\sqrt{M^2-\p^2}}\sin{\p}
  $$
  is increasing in $\big(0,\min\{M,\frac{\pi}{2}\}\big)$. This will imply non-convexity of $D_\PP(z,M)$.

  First,
  $$
    g'(\p) = e^{-\sqrt{M^2-\p^2}} \left( \cos{\p}+\frac{\p \sin{\p}}{\sqrt{M^2-\p^2}} \right)
  $$
  and this is clearly non-negative for $0 < \p < \min\{M,\frac{\pi}{2}\}$. Therefore $g(\p)$ is increasing.

  Second, by a straightforward computation we obtain
  $$
    f'(\p) = e^{-\sqrt{M^2-\p^2}} \left( \frac{\p \cos{\p}}{\sqrt{M^2-\p^2}}-\sin{\p} \right)
  $$
  and
  $$
    f''(\p) = \frac{e^{-\sqrt{M^2-\p^2}} \big( \big( M^2-\sqrt{M^2-\p^2}(M^2-2\p^2) \big) \big)\cos \p+2\p(\p^2-M^2)\sin \p}{(\sqrt{M^2-\p^2})^3}.
  $$
  Now $f'(0) = 0$ and $f''(0) = e^{-M}(1/M-1) < 0$ and therefore $f(\p)$ is concave in the neighborhood of $\p = 0$.
\end{proof}

Theorem \ref{notconvexPP} can easily be extended to the case $n \ge 3$.

\begin{corollary}\label{notconvexPS}
  If $M>1$ and $z \in \PS$, then the quasihyperbolic ball $D_{\PS}(z,M)$ is not convex.
\end{corollary}
\begin{proof}
  Let us choose any $y \in \PS$ such that $y \neq t \, z$ for all $t \in \R$. Now $D_{\PS}(z,M) \cap \span (0,y,z)$ is not convex by Theorem \ref{notconvexPP} and therefore the quasihyperbolic ball $D_{\PS}(z,M)$ cannot be convex.
\end{proof}

Let us now consider the convexity of the quasihyperbolic balls in the case $M \leq 1$ and $n = 2$.

\begin{theorem}\label{convexPP}
  For $0 < M \leq 1$ and $z \in \PP$ the quasihyperbolic disk $D_\PP(z,M)$ is strictly convex.
\end{theorem}
\begin{proof}
  Let $z=1$ and $x \in \partial D_\PP(z,M)$. By symmetry it is sufficient to consider the upper half $D$ of $\partial D_\PP(z,M)$, which is given by
  \begin{equation}\label{funcxs}
    x = x (s) = (e^s \cos \p,e^s \sin \p),
  \end{equation}
  where $M \in (0,\pi)$, $s \in [-M,M]$ and $\p = \p (s) = \sqrt{M^2-s^2}$. Now $\p'(s) = -s/\p(s)$ and therefore for $s \in (-M,M)$
  $$
    x'(s) = \frac{e^s}{\p(s)} \big( a(s),b(s) \big),
  $$
  where $a(s) = \p(s) \cos \p(s)+s \sin \p(s)$ and $b(s) = \p(s) \sin \p(s)-s \cos \p(s)$. Now $t(s) = \big( a(s),b(s) \big)$ is a tangent vector of $D$ for $s \in [-M,M]$. Equality $t(s) = 0$ is equivalent to $s^2 = -\p(s)^2$, which never holds. Since $t(s) \ne 0$ for all $s \in [-M,M]$ the angle $\alpha(s) = \arg t(s)$ is a continuous function on $(-M,M)$. We need to show that $\alpha(s)$ is strictly decreasing on $[-M,M]$.

  Since $\alpha (s) = \arctan \big( b(s)/a(s) \big)$ and $\arctan$ is strictly increasing, we need to show that $c(s) = b(s)/a(s)$ is strictly decreasing. By a straightforward computation
  \begin{equation}\label{dc}
    c'(s) = \frac{a(s)b'(s)-b(s)a'(s)}{a(s)^2} = -\frac{(1+s)M^2}{\p(s)a(s)^2}
  \end{equation}
  and the assertion follows.
\end{proof}

\begin{remark}
  The boundary $\partial D_\PP(1,M)$ is smooth since $\alpha(s)$ is continuous,
  $$
    t(M) = (0,-M) \quad \textnormal{and} \quad t(-M) = (0,M).
  $$
\end{remark}

By using the symmetry of the quasihyperbolic balls we can extend Theorem \ref{convexPP} to the case of punctured space.

\begin{lemma}\label{symmetric}
  Let the domain $G \subset \Rn$ be symmetric about a line $l$, $G \cap l \ne \emptyset$ and $G \cap L$ be strictly convex for any plane $L$ with $l \subset L$. Then $G$ is strictly convex.
\end{lemma}
\begin{proof}
  We may assume that the line $l$ is the first coordinate axis of $\Rn$ to simplify notation. Let us define function $f \colon \R \to [0,\infty)$ by
  $$
    f(x) = \left\{ \begin{array}{ll} d(x,z), & \textrm{if there exists }z=(x,z_2, \dots ,z_n) \in \partial G\\ 0, & \textrm{otherwise.} \end{array} \right.
  $$
  Since $G$ is symmetric about $l$ and $G \cap l \ne \emptyset$ there exists such $x_0,x_1 \in \R$ that $f[x_0,x_1] = [0,d]$ for $d < \infty$ and $f(x_0) = 0 = f(x_1)$. Since $G \cap L$ is convex the function $f$ is concave on $[x_0,x_1]$.

  Let $x,y \in G$, $x\ne y$ be arbitrary and denote $A_x = \{ z = (x_1,z_2, \dots ,z_n) \in G \colon d(z,l) = d(x,l) \}$ and $A_y = \{ z = (y_1,z_2, \dots ,z_n) \in G \colon d(z,l) = d(y,l) \}$. The line segment $[x,y]$ is contained in the closure of the convex hull of $A_x \cup A_y$, which is contained in $G$ by the concavity of $f$.
\end{proof}

\begin{corollary}\label{convexPS}
  For $0 < M \leq 1$ and $z \in \PS$ the quasihyperbolic ball $D_{\PS}(z,M)$ is strictly convex.
\end{corollary}
\begin{proof}
  By (\ref{moeq}) the quasihyperbolic ball $D_{\PS}(x,M)$ is symmetric about the line that contains $x$ and $0$. By Lemma \ref{symmetric} and Theorem \ref{convexPP} $D_{\PS}(x,M)$ is strictly convex for $0 < M \le 1$.
\end{proof}

%%%%%%%%%%%%%%%%%%%%%%%%%%%%%%%%%%%%%%%%%%%%%%%%%%%%%%%%%%%%%
\section{Starlikeness of quasihyperbolic balls in punctured space}

In this section we will find the maximum value of the radius $M$ for which the quasihyperbolic ball $D_\PS(x,M)$ is starlike with respect to $x$. As in the previous section we will first consider the quasihyperbolic disks in the punctured plane and then extend the results to the punctured space.

Let us define a constant $\kappa$ as the solution of the equation
\begin{equation}\label{kappa}
    \cos \sqrt{p^2-1}+\sqrt{p^2-1}\sin \sqrt{p^2-1} = e^{-1}
\end{equation}
for $p \in [1,\pi]$. The proof of the next theorem shows that the equation (\ref{kappa}) has only one solution $\kappa$ on $[1,\pi]$ with numerical approximation
$$
  \kappa \approx 2.83297.
$$

\begin{remark}
  According to \cite{am} the number $\kappa$ was first introduced by P.T. Mocanu in 1960 \cite{mon}. Later V. Anisiu and P.T. Mocanu showed \cite[page 99]{am} that if $f$ is an analytic function in the unit disk, $f(0) = 0$ and
  $$
    \left| \frac{f''(z)}{f'(z)} \right| \le \kappa,
  $$
  then $f$ is starlike with respect to 0.
\end{remark}

\begin{theorem}\label{starlikeness}
  The quasihyperbolic disk $D_\PP(x,M)$ is strictly starlike with respect to $x$ for $0 < M \le \kappa$ and is not starlike with respect to $x$ for $M > \kappa$.
\end{theorem}
\begin{proof}
  Because of symmetry we will consider $\partial D_\PP(x,M)$ only above the real axis and by the similarity it is sufficient to consider only the case $x=1$. By Theorem \ref{convexPP} we need to consider $M \in (1,\pi)$.

  Let us denote by $l(s)$ a tangent line of the upper half of $\partial D_\PP(1,M)$. The slope of the tangent line $l(s)$ is described by the function $c(s)$ defined in the proof of Theorem \ref{convexPP}. By (\ref{dc}) the function $c(s)$ is increasing on $[-M,-1]$ and decreasing on $[-1,M]$. We need to find $M$ such that $l(s)$, $s \in [-M,M]$, goes through point 1 exactly once. In other words, we need to find $M$ such that $l(-1)$ goes through 1.

  The tangent line $l(s)$ goes through 1 if and only if
  \begin{equation}\label{kappaeq}
    c(s) = \frac{x_2}{x_1-1},
  \end{equation}
  where $x_1 = e^s \cos \p(s)$ and $x_2 = e^s \sin \p(s)$. The equation (\ref{kappaeq}) in the special case $s = -1$ is equivalent to
  $$
    \frac{e \cos \sqrt{M^2-1}+e \sqrt{M^2-1} \sin \sqrt{M^2-1}-1}{(e-\cos \sqrt{M^2-1})(\sqrt{M^2-1} \cos \sqrt{M^2-1}-\sin \sqrt{M^2-1})}  = 0,
  $$
  which holds if and only if $M = \kappa$.

  We will finally show that $M = \kappa$ is the only solution of (\ref{kappa}) on $(1,\pi)$. We define function $h(x) = \cos x+x\sin x -e^{-1}$ and show that it has only one root on $(0,\sqrt{\pi^2-1})$. Since $h'(x) = x \cos x$, $h(0) 1-e^{-1} > 0$ and $h(\sqrt{\pi^2-1}) < h(11 \pi / 12) < 0$ the function $h$ has only one root on $(0,\sqrt{\pi^2-1})$ and the assertion follows.
\end{proof}

\begin{corollary}\label{kstarlikeness}
  The quasihyperbolic ball $D_{\PS}(x,M)$ is strictly starlike with respect to $x$ for $0 < M \le \kappa$ and is not starlike with respect to $x$ for $M > \kappa$.
\end{corollary}
\begin{proof}
  By Theorem \ref{starlikeness} the claim is true for $n = 2$. Let us assume $n > 2$ and choose $x \in \PS$ and $M \in (0,\kappa]$. Let us assume, on the contrary, that there exist $y \in \partial D_\PS(x,M)$ and $z \in (x,y)$ such that $z \in \partial D_\PS(x,M)$. Now $z \in \partial D_\PS(x,M) \cap \span (0,x,y)$ and therefore $D_\PP(x,M)$ is not strictly starlike with respect to $x$. This is a contradiction by Theorem \ref{starlikeness}.
\end{proof}

\begin{remark}
  Let us consider the starlikeness property of the quasihyperbolic disk $D_\PP(x,M)$ with respect to any point $z \in D_\PP(x,M)$. For $M > 1$ and $z = (e^{-M}+\e)x/|x|$, where $\e > 0$, we can choose $\e$ so small that $D_\PP(x,M)$ is not starlike with respect to $z$. On the other hand for $M < \lambda \approx 2.9648984$, where $\lambda$ is a solution of
  \begin{equation}\label{starshapedball}
    \cos \sqrt{p^2-1}+\sqrt{p^2-1}\sin \sqrt{p^2-1} = e^{-1-p},
  \end{equation}
  $D_\PP(x,M)$ is starlike with respect to $z = (e^M-\e)x/|x|$ for small enough $\e > 0$. This is also true for quasihyperbolic balls $D_\PS(x,M)$. The equation (\ref{starshapedball}) can be obtained by similar computations as in the proof of Theorem \ref{starlikeness}.
\end{remark}

\begin{remark}\label{notround}
  For $M \le \pi$ we note that
  $$
    \lim_{\p \to M} c(s) = -\infty \quad \textnormal{and} \quad \lim_{\p \to -M} c(s) = \infty
  $$
  and therefore $D_\PS(x,M)$ smooth. For $M > \pi$ the boundary $\partial D_\PS(x,M)$ is defined by (\ref{funcxs}) for $s \in [m,M]$, where $m = \max \{ t \in (-M,M) \colon \sin \sqrt{M^2-t^2} = 0 \}$. Therefore
  $$
    \lim_{\p \to M} c(s) = -\infty \quad \textnormal{and} \quad \lim_{\p \to m} c(s) = \frac{-m \cos \p(m)}{\p(m) \cos \p(m)} = -\frac{m}{\p(m)},
  $$
  where $|-m/\p(m)| < \infty$, and $D_\PS(x,M)$ is not smooth at $\big( e^m \sin \p(m),0 \big)$. Note that by (\ref{moeq}) $D_\PP(x,M)$ is not simply connected for $M > \pi$ and is simply connected for $M \in (0,\pi]$.
\end{remark}

\begin{proof}[Proof of Theorem \rm\ref{mainthm}]
  The claim is clear by Corollaries \ref{notconvexPS}, \ref{convexPS} and \ref{kstarlikeness}.
\end{proof}

The following lemma shows a property of the Euclidean radius of a quasihyperbolic ball.

\begin{lemma}
  Let  $M \in (0,\kappa]$, $z \in \PS$ and $x,y \in \partial D_\PS(z,M)$. Then $\measuredangle xz0 < \measuredangle yz0$ implies $|x-z| < |y-z|$.
\end{lemma}
\begin{proof}
  Since $M \le \kappa$ the quasihyperbolic ball $D_\PS(z,M)$ is strictly starlike with respect to $z$ by Theorem \ref{kstarlikeness} and the angle $\measuredangle xz0$ determines the point $x$ uniquely. By symmetry and similarity it is sufficient to consider only the case $n = 2$ and $z = 1$. We will show that the function
  $$
    f(s) = |x(s)-1|^2
  $$
  is strictly increasing on $(-M,M)$, where $x(s)$ defined by (\ref{funcxs}). Now
  $$
    f(s) = |x(s)|^2+1-2|x(s)| \cos p(s) = e^{2s}+1-2 e^s \cos \p(s)
  $$
  for $s \in [-M,M]$ and
  $$
    f'(s) = 2 e^s \left( e^s-\cos \p(s)-\frac{s \sin \p(s)}{\p(s)} \right).
  $$

  If $s \in (0,M)$, then
  $$
    e^s-\cos \p(s)-\frac{s \sin \p(s)}{\p(s)} \ge e^s-\cos \p(s)-s \ge e^s-1-s > 0
  $$
  and $f'(s) > 0$.

  If $s \in [-M,0)$, then $e^s-\cos \p(s)-s \sin \p(s) / \p(s) > 0$ is equivalent to $e^{-t}-\cos \p(t)+t \sin \p(t) / \p(t) > 0$ for $t \in (0,M]$. Because $M < 3$, by elementary calculus
   {\small
  \begin{eqnarray*}
    e^{-t}-\cos \p(t)+\frac{t \sin \p(t)}{\p(t)} & \ge & \left( 1-t+\frac{t^2}{2}-\frac{t^3}{6} \right) - \left( 1-\frac{\p(t)^2}{2}+\frac{\p(t)^4}{24} \right) + \left( t-t\frac{\p(t)^2}{6} \right)\\
    & = & \frac{1}{24} \left( 12M^2-M^4-4M^2t+2M^2t^2-t^4 \right) > 0
  \end{eqnarray*}
  }
  and also $f'(s) > 0$. Therefore $f$ is strictly increasing and the assertion follows.
\end{proof}

Finally we pose an open problem concerning the uniqueness of short geodesics: are quasihyperbolic geodesics with length less than $\pi$ always unique?

\medskip
\emph{Acknowledgements.} This paper is part of the author's PhD thesis, currently written under the supervision of Prof. M. Vuorinen and supported by the Academy of Finland project 8107317.

%%%%%%%%%%%%%%%%%%%%%%%%%%%%%%%%%%%%%%%%%%%%%%%%%%%%%%%%%%%%%


\begin{thebibliography}{MAO}

\bibitem [1]{am} {\sc V. Anisiu, P.T. Mocanu}:
\emph{On a simple sufficient condition for starlikeness.} Mathematica (Cluj) {\bf 31} (54) (1989), 97--101.

\bibitem [2]{f} {\sc R.H. Fowler}:
\emph{The Elementary Differential Geometry of Plane Curves.} Cambridge University Press, 1929.

\bibitem [4]{go} {\sc F.W. Gehring, B.G. Osgood}:
\emph{Uniform domains and the quasi-hyperbolic metric.} J. Anal. Math. {\bf 36} (1979), 50--74.

\bibitem [5]{gp} {\sc F.W. Gehring, B.P. Palka}:
\emph{Quasiconformally homogeneous domains.} J. Anal. Math. {\bf 30} (1976), 172--199.

\bibitem [3]{h} {\sc P. Hästö}:
\emph{Isometries of the quasihyperbolic metric.} Pacific J. Math. {\bf 230}:2 (2007), 315--326.

\bibitem [6]{l} {\sc H. Lindén}:
\emph{Quasihyperbolic Geodesics and Uniformity in Elementary Domains.} Dissertation, University of Helsinki, 2005, Ann. Acad. Sci. Fenn. Math. Diss. 146 (2005).

\bibitem [7]{mo} {\sc G.J. Martin, B.G. Osgood}:
\emph{The quasihyperbolic metric and the associated estimates on the
hyperbolic metric.} J. Anal. Math. {\bf 47} (1986), 37--53.

\bibitem [8]{mv} {\sc O. Martio, J. Väisälä}:
\emph{Quasihyperbolic geodesics in convex domains II.} Manuscript, 2006.

\bibitem [9]{mon} {\sc P.T. Mocanu}:
\emph{Sur le rayon de stellarité des fonctions univalentes.} (Romanian) Acad. R. P. Romêne Fil. Cluj Stud. Cerc. Mat. {\bf 11}, 1960, 337--341.

\bibitem [10]{v1} {\sc J. Väisälä}:
\emph{Quasihyperbolic geometry of domains in Hilbert spaces.} Ann. Acad. Sci. Fenn. Math. 32 (2007), no. 2, 559--578.

\bibitem [11]{v2} {\sc J. Väisälä}:
Private communication December 2006.

\bibitem [12]{vu1} {\sc M. Vuorinen}:
\emph{Capacity densities and angular limits of quasiregular mappings.} Trans. Amer. Math. Soc. {\bf 263} (1981), 2, 343--354.

\bibitem [13]{vu2} {\sc M. Vuorinen}:
\emph{Conformal invariants and quasiregular mappings.} J. Anal. Math. {\bf 45} (1985), 69--115.

\bibitem [14]{vu3} {\sc M. Vuorinen}:
\emph{Conformal Geometry and Quasiregular Mappings.} Lecture Notes in Math. Vol. {\bf 1319}, Springer-Verlag, 1988.

\bibitem [15]{vu4} {\sc M. Vuorinen}:
\emph{Metrics and quasiregular mappings.} Proc. Int. Workshop on Quasiconformal Mappings and their Applications, IIT Madras, Dec 27, 2005--Jan 1, 2006, ed. by S. Ponnusamy, T. Sugawa and M. Vuorinen, \emph{Quasiconformal Mappings and their Applications}, Narosa Publishing House, 291--325, New Delhi, India, 2007.

\end{thebibliography}
\end{document}